\documentclass[12pt]{amsart}
\usepackage{mathrsfs, amsmath,amssymb}
\usepackage{graphicx}
\usepackage[latin1]{inputenc}
\usepackage{amsmath}
\usepackage{amssymb}
\usepackage{amsthm}
\usepackage{amsfonts}
\usepackage{fullpage}
\newtheorem{Prop}{Proposition}

\newtheorem{lem}{Lemma}
\newtheorem{nota}{Remark}

\title{Spectral Distribution of the Free unitary Brownian motion: another approach}
\author[N. Demni]{Nizar Demni}
\address{IRMAR, Universit\'e de Rennes 1\\ Campus de
Beaulieu\\ 35~042 Rennes cedex\\ France}
\email{nizar.demni@univ-rennes1.fr}

\author[T. Hmidi]{Taoufik Hmidi}
\address{IRMAR, Universit\'e de Rennes 1\\ Campus de
Beaulieu\\ 35~042 Rennes cedex\\ France}
\email{thmidi@univ-rennes1.fr}

\begin{document}

\maketitle
\begin{abstract}
We revisit the description provided by Ph. Biane of the spectral measure of the free unitary Brownian motion. We actually construct for any $t \in (0,4)$ a Jordan curve $\gamma_t$ around the origin, not intersecting the semi-axis $[1,\infty[$ and whose image under some meromorphic function $h_t$ lies in the circle. Our construction is naturally suggested by a residue-type integral representation of the moments and $h_t$ is up to a M\"obius transformation the main ingredient used in the original proof. Once we did, the spectral measure is described as the push-forward of a complex measure under a local diffeomorphism yielding its absolute-continuity and its support. Our approach has the merit to be an easy yet technical exercise from real analysis.    
\end{abstract}

\section{Reminder and Motivation}
In his pioneering paper \cite{Biane}, Ph. Biane defined and studied the so-called free unitary or multiplicative Brownian motion. It is a unitary operator-valued L\'evy process with respect to the free multiplicative convolution of probability measures on the unit circle $\mathbb{T}$ (or equivalently the multiplication of unitary operators that are free in some non commutative probability space). Besides, the spectral distribution $\mu_t$ at any time $t \geq 0$ is characterized by its moments 
\begin{equation*}
m_n(t) :=  \int_{\mathbb{T}} z^n d\mu_t(z) = e^{-nt/2} \sum_{k=0}^{n-1}\frac{(-t)^k}{k!} n^{k-1}\binom{n}{k+1}, n \geq 1,
\end{equation*}
and $m_{-n}(t) = m_n(t), n \geq 1$ since $Y^{-1}$ defines a free unitary Brownian motion too.  This alternate sum is not easy to handle analytically since for instance if we try to work out the moments generating function of $(m_n(t))_{n \geq 1}$ 
\begin{equation*}
M_t(w) := \sum_{n\geq 1}m_n(t)w^n, \, |w| < 1, 
 \end{equation*}
 then we are led to 
\begin{equation*}
M_t(w) = u \sum_{k \geq 0} \frac{(-ut)^k}{k!}S_k(u)
\end{equation*}
where
\begin{equation*}
S_k(u) : = \sum_{n \geq 0} (n+k+1)^{k-1}\binom{n+k+1}{k+1} u^{n}, \quad k \geq 0,
\end{equation*}
and $u = we^{-t/2}$. Nonetheless, the explicit inverse function of 
\begin{equation*}
\tau_t(w) := \int_{\mathbb{T}}\frac{z+w}{z-w}\mu_t(dz) = 1+2M_t(w)
\end{equation*}
in the open unit disc played a key role in the description of $\mu_t$ (\cite{Biane0}). More precisely, it was proved there that $\mu_t$ is an absolutely continuous probability measure with respect to the normalized Haar measure on $\mathbb{T}$ and that its density is a real analytic function inside its support. The latter coincides with $\mathbb{T}$ when $t > 4$ while it is given by the angle 
\begin{equation*}
|\theta| \leq \beta(t) := (1/2)\sqrt{t(4-t)} + \arccos(1-(t/2))
\end{equation*}
when $t \leq 4$. When proving these important results, the author relied on free stochastic integration (Lemma 11 p.266), Caratheodory's extension Theorem for Riemann maps (Lemma 12 p.270) and a Poisson-type integral representation for this kind of maps (see the proof of Proposition 10 p.270). In the present paper, we shall recover Biane's results from more simpler considerations than the ones used in the original proof. Indeed, for $t \in (0,4)$, there exists a unique piecewise smooth Jordan curve $\gamma_t$ around the origin, not intersecting the semi-axis $[1,\infty[$ and whose image under some function $h_t$ lies in $\mathbb{T}$. Our construction is naturally suggested by a residue-type integral representation of $m_n(t)$ and fails when $t \geq 4$. Note  that the same phenomenon happens here and in Biane's proof: $\gamma_t$ is constructed upon two curves that have a non empty intersection if and only if $t < 4$, while the inverse function of $\tau_t$ is defined on the interiors of two Jordan domains whose boundaries have the same phase transition (\cite{Biane0} p. 267). Moreover, the function $h_t$ appears in the integrand of our residue-type representation and coincides up to the
M\"obius transformation 
\begin{equation*}
z \mapsto \frac{2}{z}-1
\end{equation*}
with the inverse function of $\tau_t$ used in the original proof. Once the curve $\gamma_t$ is constructed, we consider a piecewise smooth parametrization $z_t$ of $\gamma_t$ and prove that the derivative of $\theta \mapsto \arg[h_t(z_t(\theta))]$ vanishes if and only if $\theta = \pm \arccos(\sqrt{t}/2)$ (note that similarly the derivative of $\tau_t^{-1}$ vanishes if and only if $z = \pm i \sqrt{(4/t) -1}$, \cite{Biane0} p. 269). As a matter of fact, $\theta \mapsto \arg[h_t(z_t(\theta))]$ defines far from the critical points a local diffeomorphism so that performing local changes of variables in our integral representation yields both the absolute-continuity of $\mu_t$ with respect to the Haar measure on $\mathbb{T}$ and the description of its support.

\section{A Residue-type representation of $m_n(t)$}
A residue-type integral representation of $m_n(t), n \geq 1$ is not new in its own. Indeed, the one we derive below is an elaborated version of the one used in order to determine the decay order of $m_n(t)$ (see \cite{Levy} p. 566): 
\begin{equation*}
m_n(t) = \frac{e^{-nt/2}}{2i\pi n}\int_{\gamma} e^{-ntz}\left(1+\frac{1}{z}\right)^n dz
\end{equation*}
where $\gamma$ is a circle around the origin. More precisely, we need to integrate by parts since then the meromorphic integrand we obtain determines in a unique way the required curve $\gamma_t$ we mentioned above. 
To proceed, we first apply Cauchy's Residues Theorem to $z \mapsto z^ke^{n/z}$ in order to get
\begin{equation*}
\frac{n^{k+1}}{(k+1)!} = \frac{1}{2i\pi} \int_{\gamma} z^{k}e^{n/z}dz,
\end{equation*}
then use the combinatorial identity for binomial numbers: 
\begin{equation*}
\binom{n}{k+1} = \frac{n}{k+1}\binom{n-1}{k}, \quad n \geq 1.
\end{equation*}
As a result, for any $n \geq 1$
\begin{align*}
m_n(t) &= e^{-nt/2} \sum_{k=0}^{n-1}\frac{(-t)^k}{k!} n^{k-1}\binom{n}{k+1}
\\& = \frac{e^{-nt/2}}{2i\pi n} \int_{\gamma} \sum_{k=0}^{n-1} \binom{n-1}{k}(-tz)^k e^{n/z} dz \\&
=  \frac{1}{2i\pi n} \int_{\gamma} (1-tz)^n e^{n(1/z - t/2)} \frac{dz}{1-tz} 
\\& = \frac{1}{2i\pi n} \int_{t\gamma} \left[(1-z)e^{t(1/z - 1/2)}\right]^n \frac{dz}{t(1-z)}
\\& = \frac{1}{2i\pi n} \int_{\gamma} \left[(1-z)e^{t(1/z - 1/2)}\right]^n \frac{dz}{t(1-z)}
\\& := \frac{1}{2i\pi n} \int_{\gamma} [h_t(z)]^n \frac{dz}{t(1-z)}.
\end{align*}
Now, choose further $\gamma$ such that it does not meet the semi-axis $[1,\infty[$ then $z \mapsto \log(1-z)$ is well defined for $z \in \gamma$ and is holomorphic there. Hence, setting $z = re^{i\theta}, 0 < r < 1$ and integrating by parts yield  
\begin{align*}
m_n(t) & = \frac{1}{2i\pi\,t} \int_{\gamma} [h_t(z)]^n \frac{h_t'(z)}{h_t(z)}\log(1-z) dz. 
\end{align*}
As a matter of fact, $m_n(t)$ is the residue of 
\begin{equation*}
z \mapsto \frac{1}{t} [h_t(z)]^n\frac{h_t'(z)}{h_t(z)}  \log(1-z)
\end{equation*}
at $z=0$ so that one may integrate along any piecewise smooth Jordan curve $\gamma_t$ (possibly depending on $t$) around zero provided that the integrand is well defined.  Assume further that we can choose  
$\gamma_t$ such that $|h_t(\gamma_t)| \in \mathbb{T}$ and let $\theta \in [-\pi,\pi] \mapsto z_t(\theta)$ be a piecewise smooth parametrization of $\gamma_t$, then 
\begin{align*}
m_n(t) &=  \frac{1}{2i\pi\,t} \int_{-\pi}^{\pi} e^{in\arg h_t(z_t(\theta))}\frac{h_t'(z_t(\theta))}{h_t(z_t(\theta))}\log(1-z_t(\theta))z_t'(\theta)d\theta 
\\& = \int_{-\pi}^{\pi} e^{in\theta}\nu_t(d\theta)
\end{align*}
where $\nu_t$ is the push-forward of 
\begin{equation*}
\frac{h_t'(z_t(\theta))}{h_t(z_t(\theta))}z_t'(\theta) \log(1-z_t(\theta)){\bf 1}_{[-\pi,\pi]}(\theta)\frac{d\theta}{2i\pi}  
\end{equation*}
under the map $\theta \mapsto \arg h_t(z_t(\theta))$. Heuristically, we are attempted to conclude that $\nu_t = \mu_t$ however it is not clear at all that $\nu_t$ is a real measure and there is no guarantee even for $\gamma_t$ to exist. 
Below, we shall prove that $\gamma_t$ exists if and only if $t \in (0,4)$ and is unique. Then $\gamma_t$ splits into two curves $\gamma_t^1 \cup \gamma_t^2$ where $\theta \mapsto \arg[h_t(z_t(\theta))]$ is a diffeomorphism from $\gamma_t^i$ to its image, for $i=1,2.$ As a result, a change of variables shows that $\nu_t$ and $\mu_t$ coincide since their trigonometric moments do, therefore $\mu_t$ is absolutely continuous and its support is easily recovered as $h_t(\gamma_t)$.

\section{Construction of the curve  $\gamma_t$}
Our main result is stated as 
\begin{Prop}
Let $t\in (0,4),$ then there exists a unique (piecewise smooth) Jordan curve $\gamma_t$ such that 
\begin{itemize}
\item $h_t(\gamma_t) \in \mathbb{T}$.
\item $\gamma_t$ encircles $z=0$ and $\gamma_t \cap [1,\infty[ \, = \, \emptyset$.  
\end{itemize}  
\end{Prop} 

\begin{proof}
Before coming into computations, let us point to the fact that $\gamma_t$ has to be invariant under complex conjugation: this fact follows from $|h_t(\overline{z})| = |\overline{h_t(z)}| = |h_t(z)|$. It is also coherent with the fact that $\rho_t$ shares the same invariance property since $Y^{-1} = Y^{\star}$ is also a free unitary Brownian motion, therefore we only consider $\theta \in [0,\pi]$. 
We also inform the patient reader that both polar and cartesian coordinates are used in the sequel depending on how behaves the curve defined by $|h_t(z)| = 1$ when $t$ runs over $]0,4[$. 
\subsection{Polar coordinates} Let $z = re^{i\theta}, \theta \in [0,\pi]$ then $|h_t(z)| = 1$ is equivalent to
\begin{equation*}
g_{t,\theta}(r) : = (1+ r^2 - 2r\cos\theta) e^{(2t\cos \theta)/r}= e^t.
\end{equation*}
We distinguish two regions:

$\bullet$ \underline{$\{\theta, \cos\theta<0\}$:} 
In this region the function $g_{t,\theta}$ is increasing  and satisfies  
\begin{equation*}
\lim_{r \rightarrow 0^+} g_{t,\theta}(r) = 0, \quad   \lim_{r \rightarrow \infty} g_{t,\theta}(r) = \infty. 
\end{equation*}
The monotonicity of $g_{t,\theta}$ follows obviously from 
\begin{equation}\label{E1}
g_{t,\theta}'(r) = 2e^{2t\cos \theta/r}\left(r- \cos\theta - (1+r^2 - 2r\cos\theta)\frac{t\cos\theta}{r^2}\right).
\end{equation}
Then $g_{t,\theta}(r) = e^t$ has a unique solution $r = r_t(\theta)>0$. Note that the implicit function Theorem together with the fact that $g_{t,\theta}'(r) > 0$ show that $\theta \mapsto r_{t}(\theta)$ is at least $C^1$ on $]\pi/ 2,\pi[$. 
Now, it is obvious that $r_t$ does not vanish on $\{\cos\theta < 0\}$ and we shall check that it remains so on $\{\cos\theta = 0\}$. More precisely, we claim that $r_t(\theta) > \sqrt{t}$ which may be proved as follows: use
\begin{equation*}
 1+t<e^t\quad \hbox{and} \quad 1-2\sqrt{t}\cos\theta<e^{-2\sqrt{t}\cos\theta}
\end{equation*}
to get
\begin{eqnarray*}
1+t-2\sqrt{t}\cos\theta \leq (1+t)(1-2\sqrt{t}\cos\theta) < e^{t-2\sqrt{t}\cos\theta}.
\end{eqnarray*}
This in turn yields
$$
g_{t,\theta}(\sqrt{t})<e^t
$$
and the monotonicity of $g_{t,\theta}$ proves the claim. As a matter of fact $r_t$ extends continuously to $\pi/2$ and one obviously has
\begin{equation*}
r_t(\pi/2) = \sqrt{e^t-1}.
\end{equation*}

$\bullet$ {\underline{$\{\theta,\cos\theta>0$\}:}} 
For these values of $\theta$, observe that  $g_{t,\theta}(2\cos\theta) = e^t$ for all $t$. However, $r_t(\theta) = 2\cos\theta$ does not fulfill our requirements since on the one hand it vanishes at $\theta = \pm \pi/2$ and on the other hand it meets $[1,\infty[$ at $\theta=0$. Fortunately, there exists another radius $r_t$ satisfying $g_{t,\theta}(r_t) = e^t$, $r_t(\pi/2) \neq 0$ and $r_t(0) \in ]0,1[$. Indeed, letting $\theta_t := \arccos(\sqrt{t}/2) \in ]0,\pi/2$, then 
\begin{equation*}
g_{t,\theta}'(2\cos \theta) = \frac{4\cos^2\theta - t}{2\cos \theta}e^t  
\end{equation*}
is negative on $]\theta_t,\pi/2[$, positive on $[0,\theta_t[$ and 
\begin{equation*}
\lim_{r \rightarrow \infty} g_{t,\theta}(r) = + \infty 
\end{equation*}
for any $\theta$ such that $\cos \theta \geq 0$. Besides, this radius is unique except possibly for $2+\sqrt{3} < t < 4$ and for $\theta$ close to zero, and we keep using polar coordinates. For exceptional values of $(t,\theta)$, cartesian coordinates are more adequate and doing so we recover $\gamma_t$ as the graph of some function.

$\blacktriangle$ \underline{$\{ 0< t \leq 2+\sqrt{3}$\}}:
When $t\in]0,1]$, we  shall prove that $g_{t,\theta}$ is a convex function. To this end, we compute the second derivative of $g_{t,\theta}$ 
\begin{align*}
g_{t,\theta}''(r) &= e^{2t\cos \theta/r}\left[\left(\frac{4t\cos\theta}{r^3} + \frac{4t^2\cos^2\theta}{r^4}\right)(1+r^2-2r\cos\theta) + 2 - \frac{8t\cos\theta}{r^2}(r-\cos\theta)\right] \\
            &= \frac{e^{2t\cos \theta /r}}{r^4}\left[2r^4 - 4tr^3\cos\theta + 4t^2r^2\cos^2\theta + 4rt\cos\theta(1-2t\cos^2\theta) +4t^2\cos^2\theta \right].
\end{align*}
The first equality shows that if $r \leq \cos\theta$ then $g_{t,\theta}''(r) > 0$. For $r \geq \cos\theta$, define $k_t$ by 
\begin{equation*}
g_{t,\theta}''(r) = \frac{e^{2t\cos \theta /r}}{r^4}k_{t,\theta}(r).
\end{equation*}
Then 
\begin{eqnarray*}
k_{t,\theta}'(r) &=& 8r^3 - 12tr^2\cos\theta + 8t^2r\cos^2\theta + 4t\cos \theta(1-2t\cos^2\theta) \\
k_{t,\theta}''(r) &=& 8(3r^2-3tr\cos \theta + t^2\cos^2\theta) \geq 0
\end{eqnarray*}
even for all $r \geq 0$. Hence  
\begin{align*}
k_{t,\theta}'(r) \geq k_{t,\theta}'(t\cos\theta) &= 4t\cos\theta(t^2\cos^2\theta -2 t\cos^2\theta+1) 
\\& = (t\cos \theta - 1)^2 + 2t\cos\theta(1-\cos\theta) \geq 0.
\end{align*}
Since $t\cos \theta \leq \cos \theta$ when $t \leq1$, then $k_{t,\theta}'(r) \geq 0$ for all $r \geq 0$ therefore 
\begin{equation*}
k_{t,\theta}(r) \geq k_{t,\theta}(0) = 4t^2\cos^2\theta > 0
\end{equation*}
which yields the strict convexity of $g_{t,\theta}$ for $t \in [0,1]$. Now since 
\begin{equation*}
\lim_{r\to 0}g_{t,\theta}(r) = \lim_{r\to +\infty}g_{t,\theta}(r)=+\infty,
\end{equation*}
then the equation $g_{t,\theta}(r)=e^t$ admits exactly two solutions among them the trivial one $r = r(\theta) = 2\cos\theta$ which has to be discarded. The required curve $\gamma_t$ is then constructed upon the non trivial solution and defines even a $C^1$-piecewise curve. The last claim is obvious for the regular points of $g_{t,\theta}$ by the virtue of the implicit function Theorem again. So we need to focus on the critical points of the curve: 
$g_{t,\theta}'(r(\theta))=0$. But, the strict convexity of $g_{t,\theta}$ forces then $r_t(\theta)=2\cos\theta$ which gives after substituting in $g_{t,\theta}'$ the unique critical point $\theta = \theta_t$ for which $r_{t}(\theta_t)=\sqrt{t}.$
Before considering the range $t\in[1,2+\sqrt3]$, we point out that $\gamma_t, t \in ]0,4[$ crosses the positive real semi-axis at some point $0<x_t<1$ that is described in Lemma \ref{lemme1} below.
Now, let $t\in[1,2+\sqrt3]$ and define 
\begin{equation*}
v_{t,\theta}(r) := r^3 - (t+1)\cos \theta r^2 + 2t\cos^2\theta r - t\cos \theta 
\end{equation*}
so that 
\begin{equation*}
g_{t,\theta}'(r) = \frac{e^{2t\cos\theta/r}}{r^2}v_{t,\theta}(r). 
\end{equation*}
Then the first derivative of $v_{t,\theta}$ reads 
\begin{equation*}
v_{t,\theta}'(r) = 3r^2 - 2(t+1)\cos\theta r + 2t\cos^2\theta.
\end{equation*}
The discriminant of this second degree polynomial is easily computed as $4(t^2 -4t+1)\cos^2\theta$ and is easily seen to be negative on $t\in[1,2+\sqrt3]$. Consequently $v_{t,\theta}^\prime\geq 0$ thereby $v_{t,\theta}(r) \geq v_{t,\theta}(0) = -t\cos \theta$ for any $r \geq 0$. Finally, there exists $r_0 = r_0(t,\theta)$ such that $g_{t,\theta}'(r_0) = 0$ so that the variations of $g_{t,\theta}$ are described by 
\\
\\
 \centerline{$
\begin{array}{|c|ccccr|}
\hline
r     & 0 &    & r_0  &     & +\infty  \\
\hline
v_{t,\theta}(r) &          &  - & 0  & + &   \\
\hline
 &  & &  & & \\       
g_{t,\theta}(r) &  +\infty&\searrow   &  &\nearrow  &+\infty \\   
&  & & & &  \\          
\hline
\end{array}
$
}
\\
\\
The conclusion follows in similar way to the previous range of times $t\in]0,1].$ 

$\blacktriangle$ \underline{$2+\sqrt{3} < t < 4$}: 
This part of the proof needs more care since for small amplitudes of $\theta$, $\theta\mapsto r_t(\theta)$ may be  a multi-valued function (this is seen from computer-assisted pictures). This multivalence happens precisely in the interior of the region bounded by the curve $\theta\mapsto 2\cos\theta$ or $0 \leq \theta < \theta_t$, and cartesian coordinates are more adequate for our purposes. Nonetheless, we shall keep use of  polar coordinates outside the latter curve where the existence of a required radius is ensured by the same arguments evoked above. It then remains to prove uniqueness on $]2\cos \theta,+\infty[, \theta \in ]\theta_t,\pi/2]$. To this end, one easily sees that the largest root of $v_{t,\theta}'(r)$ is given by
\begin{equation*}
R_t^{+}(\theta) := \frac{t+1 + \sqrt{t^2-4t+1}}{3} \cos \theta
\end{equation*}
and that $t \mapsto R_t^+(\theta)$ is increasing. Thus, $R_t^+(\theta)\le R_4(\theta)=2\cos\theta$ yielding $v_{t,\theta}^\prime(r)>0$ for $r>2\cos\theta.$ This entails
\begin{equation*}
v_{t,\theta}(r) \geq v_{t,\theta}(2\cos\theta)=4\cos\theta(\cos^2\theta-{t}/{4})
\end{equation*}
which is negative on $]\theta_t,\pi/2]$. Therefore the variations of $g_{t,\theta}$ are summarized below
\\
\\
\centerline{$
\begin{array}{|c|ccccr|}
\hline
r     & 2\cos\theta &    & r_0  &     & +\infty  \\
\hline
v_{t,\theta}(r) &          &  - & 0  & + &   \\
\hline
 &  & &  & & \\       
g_{t,\theta}(r) &  e^t&\searrow   &  &\nearrow  &+\infty \\   
&  & & & &  \\          
\hline
\end{array}
$
}
\\
\\
for some $r_0 = r_0(t,\theta) > 2\cos \theta$, whence the uniqueness follows. Hence, the obtained branch of $\gamma_t$ is smooth and its endpoints are $i\sqrt{e^t-1}$ and $\sqrt{t}\,e^{i\arccos(\sqrt{t}/2)}$ . 

\begin{nota}
The equation $g_{t,\theta}(r)=e^t$ has no solution in the region $\big\{r>2\cos\theta,\, 0\le\theta<\theta_t \big\}.$ Indeed, $v_{t,\theta}(r) \geq v_{t,\theta}(2\cos\theta)>0$ so that $g_{t,\theta}$ is increasing.
\end{nota}

\subsection{cartesian coordinates}  

The curve $\theta\mapsto2\cos \theta, \theta\in[0,\pi/2]$ is the graph of the function $x\mapsto y = \sqrt{x(2-x)}, x \in [0,2]$. Hence, we shall restrict ourselves to the region 
\begin{equation*}
\Big\{0\le y<\sqrt{x(2-x)}, 0 < x < 2\Big\}. 
\end{equation*}
Besides, the branch of $\gamma_t$, if it exists, would meet the axis $\{y=0\}$ at a solution of  
\begin{equation*}
k_t(x):=(x-1)^2e^{t[(2/x)-1]} = 1, 0<x<2,
\end{equation*}
whose needed properties are collected in the following Lemma.    
\begin{lem}\label{lemme1}
For every $ t\in]0,4[$, the above equation admits a unique solution $x_t \in ]0,1[$ and the map $t\mapsto x_t$ is increasing. In particular $x_t > 3-\sqrt{5}$. 
\end{lem}
\begin{proof} 
Let $t \in ]0,4[$ then 
\begin{equation*}
k_t'(x)=\frac{2(x-1)(x^2-tx+t)}{x^2}e^{t[(2/x)-1]}
\end{equation*}
and the polynomial $x^2-tx+t$ is obviously positive since its discriminant is negative. As a result, we get the variations of $k_t$\\
\\
\centerline{$
\begin{array}{|c|ccccr|}
\hline
x    & 0   & & 1     & & 2 \\
\hline
k_t'(x) &&-    &&   + &  \\
\hline &&  & &  & \\       
k_t(x) & +\infty &\searrow 0  && \nearrow   &1\\   
&  & & &&   \\         
\hline
\end{array}
$
}
\\
\\
This asserts the existence of a unique value $x_t\in ]0,1[$ solving the equation $k_t(x)=1$. For the variations of $t\mapsto x_t$, we write
\begin{equation*}
x_t^\prime=-\frac{\partial_t k_t(x_t)}{\partial_x k_t(x_t)}=\frac{x_t(1-x_t)(2-x_t)}{2(x_t^2-tx_t+t)} > 0
\end{equation*}
so that $x_t > x_3$ and the Lemma is proved by noting that $k_3(3-\sqrt5) = (\sqrt{5} - 1)^2 e^{3(1+\sqrt{5})/2} \geq 1$.
\end{proof}

Now, we rewrite $g_{t,\theta}(r) = e^t$ using cartesian coordinates as 
\begin{equation*}
(1+x^2+y^2-2x)e^{2tx/(x^2+y^2)}=e^t
\end{equation*}
and denote $g_{t,x}(y)$ the LHS. In this way, $g_{t,x}(0) = k_t(x)e^t \leq e^t$ for $x \in [x_t,2]$  since $k_t(x) < 1$ for $x \in [x_t,2]$, while $g_{t,x}(0) > e^t$ for $x \notin [0,x_t[$.  
We shall prove that for each $x\in [x_t,t/2] \subset [x_t,2[$, the equation $g_{t,x}(y) = e^t$ admits a unique solution $\Big\{0\le y<\sqrt{x(2-x)}\Big\}$ while it has none when $x \notin ]x_t,t/2[$. This in turn finishes the proof of our main result. 
To proceed, we first compute
\begin{eqnarray*}
g_{t,x}'(y)&=&\frac{2y}{(x^2+y^2)^2}\Big(y^4+(2x^2-2x t)y^2+x^4-2x^3t+4x^2 t-2x t\Big)e^{2tx/(x^2+y^2)}\\
&:=& \frac{2y}{(x^2+y^2)^2}e^{2tx/(x^2+y^2)} w_{t,x}(y^2).
\end{eqnarray*}
The discriminant of $w_{t,x}(y)$ is given by $4tx(2+(t-4)x)$ and is positive since $t\geq 3$ and $x\leq 2$. Thus  $w_{t,x}(y)$ has two roots 
\begin{equation*}
y_t^\pm:=x(t-x)\pm\sqrt{tx(2+(t-4)x)}
\end{equation*}
satisfying  the following properties:

\begin{lem}\label{encadr} Let $2+\sqrt{3} < t < 4$, then:
\begin{enumerate}
\item For $x\in [x_t,2],$ then\quad $ 0 < y_t^{-}(t)\le y_t^+.$
\item  For $x\in[0,t/2]$ then\quad $y_t^{-}(t)\leq x(2-x) \leq  y_t^+$.
\item For $x\in [t/2,2]$, then\quad $x(2-x)\le y_t^{-} \leq  y_t^+$
\end{enumerate}
\end{lem}
\begin{proof}
Since $x \in ]0,2[$, then $y_t^+\geq0$ and the first property $(1)$ is equivalent to 
$$
y_t^+ y_t^- = x^4 -2x(x-1)^2 t > 0.
$$
But for any $x > 0$, the function $t\mapsto x^4-2x(x-1)^2 t$ is decreasing for positive $t$ therefore 
$$
x^4 -2x(x-1)^2 t > x^4 -8x(x-1)^2 = x(x-2)\big( x-(3-\sqrt5)\big)\big( x-(3+\sqrt5)\big)
$$
for any $t < 4$. Consequently $x^4 -2x(x-1)^2 t > 0$ for $x\in[3-\sqrt 5,2]$ in particular for $x \in [x_t,2]$ since $x_t > 3-\sqrt{5}$ by the virtue of Lemma \ref{lemme1}.  
\\
Since $y_t^+ > x(2-x)$, then the second property $(2)$ is equivalent to 
$$
x^2(t-2)^2 \leq tx(2+(t-4)x)
$$
which is in turn equivalent to $t-2x\geq0$. We are done.
\end{proof}
It remains to discuss the variations  of  $g_{t,x}$ according to:
\\
{ \it{$\star$ $x\in [x_t,t/2]$:}} Using properties $(1)$ and $(2)$ stated in Lemma \ref{encadr} we get
\\
\\
 \centerline{$
\begin{array}{|c|ccccr|}
\hline
y     & 0   & & {(y_t^-)}^{\frac12}     &   {(y_t^+)^{\frac12}}   & +\infty  \\
\hline
g_{t,x}'(y) &&+    & -  &  + &  \\
\hline &&  &&   & \\       
g_{t,x}(y) & g_{t,x}(0) &\nearrow   & \searrow &\nearrow  &+\infty \\   
&  & & & &  \\         
\hline
\end{array}
$
}
\\
\\
Since $g_{t,x}(0) \leq e^t,$  then the equation $g_{t,x}(y)=e^t$ has a unique solution $y_{t}(x)$ lying in the interval  $[0,\sqrt{x(2-x)}[$. This allows to construct a curve $x\mapsto y_t(x)$ for $x\in [x_t,t/2]$  which is continuous since the function $g_{t,x}$ depends continuously on the parameter $x$. By the virtue of the implicit function Theorem, it is even at least $C^1$-piecewise curve  since the derivative $g_{t,x}^\prime(y)$ vanishes only in a finite set. 
\\
{ $\star$ \it{$x\in [t/2,2]$:}} Using properties $(1)$ and $(3)$ of Lemma \ref{encadr}, we conclude that $g_{t,x}$ is increasing on $[0,\sqrt{x(2-x)}]$.   \mbox{Since $g_{t,x}(\sqrt{x(2-x)})=e^t$} then the equation $g_{t,x}(y)=e^t$ has no solution in $[0,\sqrt{x(2-x)}[.$
\\
{ $\star$ \it{$x\in ]0, x_t[$:}} Since $x_t \in ]0,1[$ and $t/2 > 1+(\sqrt{3}/2) > x_t$, then we make use of property $(2)$ of Lemma \ref{encadr}. But the issue depends on whether or not $y_t^-$ is positive. Assume $y_t^-$ is negative then $g_{t,x}$ is decreasing on $[0,\sqrt{x(2-x)}[$ and thus the equation $g_{t,x}(y)=e^t$ has no solution in this interval. Otherwise $y_t^->0$ and $g_{t,x}$ keeps the same variations as in the range $x\in[x_t,t/2]$:\\
\\ 
 \centerline{$
\begin{array}{|c|ccccr|}
\hline
y     & 0   & & {(y_t^-)}^{\frac12}     & & \sqrt{x(2-x)}  \\
\hline
g_{t,x}'(y) &&+    &&   - &  \\
\hline &&  & &  & \\       
g_{t,x}(y) & g_{t,x}(0) &\nearrow  && \searrow   &e^t\\   
&  & & &&   \\         
\hline
\end{array}
$
}
\\
\\
We remark that  $g_{t,x}(0)= e^t\,k(t,x)$ and  according to the variations of $x\mapsto k(t,x)$ we get $k(t,x)>1$ for $0<x<x_t.$ Consequently  the equation  $g_{t,x}(y)=e^t$ has no solution in $]0, x_t[.$\\
Finally, the above discussion shows the set  $\big\{(x,y), g_{t,x}(y)=e^t, x\in]0,2[, 0\le y\le\sqrt{x(2-x)}\big\}$ is described by a unique $C^1$-piecewise graph joining the points $(x_t,0)$ and $z = \sqrt{t}e^{i\arccos(\sqrt{t}/2)}.$
\end{proof}

\begin{nota}
For $t=4$, $g_{t,\theta}'(2\cos \theta) = 0$ if and only if $\theta= \theta_t=0$. Thus both curves whose radii solve $g_{t,\theta}(r) = e^t$ meet at $\theta =0$ thereby satisfy $r_t(0) = 2 > 1$. When $t > 4$, they even become disconnected.   
\end{nota}

\section{Critical points of $h_t$}
Let $z_t$ be a piecewise smooth parametrization of $\gamma_t$ and consider $\theta \mapsto \arg[h_t(z_t(\theta))], \theta \in [-\pi,\pi]$. Using the invariance of $\gamma_t$ under complex conjugation, we restrict our attention to $\theta \in [0,\pi]$. If $\theta \in \{0,\pi\}$ then $\arg(1-z) = 0, z \in \gamma_t$ since $r_t(0) \in ]0,1[$ therefore $\arg[h_t(z_t(0))] = 0$. Thus we discard these two values and consider $\theta \in (0,\pi)$. Then $\arg(1-z) \in (-\pi,0)$ and we need to look for critical points of 
\begin{equation*}
\arg[h_t(z)] =  \cot^{-1} \left[\frac{r\cos\theta -1}{r\sin\theta}\right] - \pi - \frac{t}{r}\sin\theta
\end{equation*}
under the constraint $z = z_t  = r_t(\theta)e^{i\theta} \in \gamma_t \cap \mathbb{C}^+$. For ease of notations, we shall omit the dependence on $t$ of the radius of $\gamma_t, t \in ]0,4[$ and write simply $r_\theta$. Hence  
\begin{align*}
\frac{d}{d\theta}\arg[h_t(r_\theta e^{i\theta})] = \frac{r_\theta^2 - r_\theta\cos \theta -r_\theta'\sin\theta}{r_\theta^2 - 2r_\theta\cos \theta+1} - t \frac{r_\theta\cos\theta - r_\theta'\sin\theta}{r_\theta^2}
\end{align*}
which vanishes if and only if 
\begin{eqnarray*}
r_\theta\Big[r_{\theta}^3 - r_{\theta}^2\cos\theta - t\cos\theta(r_{\theta}^2 - 2r_{\theta}\cos \theta+1)\Big] = r_{\theta}'\sin\theta\Big[(r_{\theta}^2 - t(r_{\theta}^2- 2r_{\theta}\cos\theta+1)\Big].  
\end{eqnarray*}
By the virtue of \eqref{E1}, the LHS may be written  as 
\begin{equation*}
\frac{1}{2}r_{\theta}^3\partial_r(g_{t,\theta})(r_{\theta}) e^{-(2t\cos\theta)/r_{\theta}}
\end{equation*}
while 
\begin{equation*}
\partial_{\theta}(g_{\theta})(r_{\theta}) = \frac{2\sin \theta}{r_{\theta}} \Big[r_{\theta}^2 - t(r_{\theta}^2- 2r_{\theta}\cos\theta+1)\Big]e^{(2t\cos\theta)/r_{\theta}}.
\end{equation*}
It follows that the following equality holds at any critical point 
\begin{equation*}
r_{\theta}^2\partial_r(g_{t,\theta})(r_{\theta}) = r'(\theta)\partial_{\theta}(g_{t,\theta})(r_{\theta}). 
\end{equation*}
Now comes the constraint $g_{t,\theta}(r_{\theta}) = e^t$ that we shall differentiate with respect to $\theta$ to get
\begin{equation*}
r'(\theta)\partial_r(g_{t,\theta})(r_{\theta}) + \partial_{\theta}(g_{t,\theta})(r_{\theta}) = 0. 
\end{equation*}

Both identities yield 
\begin{align*}
-r_{\theta}^2[\partial_r(g_{t,\theta})]^2(r_{\theta}) &= [\partial_{\theta}(g_{t,\theta})]^2(r_{\theta}).
\end{align*}
Since $r_\theta\neq0$ then  a critical value $\theta$ must satisfy 
\begin{equation*}
\partial_r(g_{t,\theta})(r_{\theta}) = \partial_{\theta}(g_{t,\theta})(r_{\theta}) = 0
\end{equation*}
which can not occur unless
\begin{equation*}
r_{\theta} = 2\cos\theta.
\end{equation*}
As a result $\theta = \theta_t$ and one easily derives using $2\cos\theta_t = \sqrt{t}$
\begin{align*}
\arg[h_t(2\cos\theta_t e^{i\theta_t})] & =  2\theta_t -\pi- \frac{1}{2}\sqrt{t(4-t)} 
\\& =   2\arccos \frac{\sqrt{t}}{2}-\pi - \frac{1}{2}\sqrt{t(4-t)}.
 \end{align*}
 Finally 
 \begin{equation*}
 \cos\left[2\arccos \frac{\sqrt{t}}{2}-\pi\right] = - \cos\left[2\arccos \frac{\sqrt{t}}{2}\right] = 1-\frac{t}{2}
 \end{equation*}
 whence we deduce that  
\begin{align*}
\arg[h_t(2\cos\theta_t e^{i\theta_t})] =  -\arccos\left(1-\frac{t}{2}\right) - \frac{1}{2}\sqrt{t(4-t)} = -\beta(t).
 \end{align*}
A similar analysis shows that $\theta \mapsto \arg[h_t(z_t(\theta))]$ has a unique critical point when $z_t \in \gamma_t \cap \mathbb{C}^-$. It is precisely given by $\theta = -\theta_t$ and 
\begin{align*}
\arg[h_t(2\cos\theta_t e^{-i\theta_t})] = \beta(t).
\end{align*}

We now proceed to the description of $\mu_t, t \in (0,4)$. 

\section{Description of $\mu_t, t \in (0,4)$} 
We have already seen that there are exactly two critical points of $\theta \mapsto \arg[h_t(z_t(\theta))], \theta \in [-\pi,\pi]$. This fact leads easily to: 
\begin{Prop}\label{curve}
There exists a partition $\gamma_t=\gamma_t^1\cup\gamma_t^2$ with 
\begin{equation*}
\gamma_t^1\subset\{z,|z-1|\leq1\}, \gamma_t^2\subset\{z,|z-1|\geq1\} 
\end{equation*}
and  such that the maps $h_{t,1}\equiv h_t:\gamma_t^1\to h_t(\gamma_t)$ and $h_{t,2}\equiv h_t:\gamma_t^2\to h_t(\gamma_t)$ are diffeomorphisms. Moreover, let $e^{i\phi}\in h_t(\gamma_t)$ then the equation 
$h_t(z)=e^{i\phi}, z \in \gamma_t$ has exactly two solutions  given by $z\in \gamma_t^1$ and 
\begin{equation*}
\frac{\overline{z}}{\overline{z}-1}\in \gamma_t^2.
\end{equation*}
\end{Prop}

\begin{proof}
The curves $\gamma_t^1$ and $\gamma_t^2$ are given by 
\begin{eqnarray*}
\gamma_t^1 & = & \big\{z\in \gamma_t, |z-1|\le1\big\} \\ 
\gamma_t^2 & = & \big\{z\in \gamma_t, |z-1|\geq1\big\}.
\end{eqnarray*}
It is clear that the critical points are located in the circle $\{z,|z-1|=1\}$ and therefore they are the end points of the curves $\gamma_t^1$ and $\gamma_t^2$. Therefore by the previous analysis of $\theta \mapsto \arg[h_t(z_t(\theta))]$ we deduce that  $h_{t,1},h_{t,2}$ are  diffeomorphisms. 
Finally, for any $z \in \gamma_t$
\begin{equation*}
h_t\left[\frac{\overline{z}}{\overline{z}-1}\right] = \frac{1}{\overline{h_t(z)}} = h_t(z). 
\end{equation*}  
We have used in the last identity the fact that $h_t(z)\in \mathbb{T}.$ We point out that the m\"{o}bius  transform $z\mapsto \frac{z}{z-1}$ is  a bijective map from $\gamma_{t,1}$ to $\gamma_{t,2}.$
\end{proof}
Thus we obviously have\footnote{$\gamma_t$ is parametrized from $\theta_t$ to $2\pi-\theta$ counter-clockwise.}  
\begin{eqnarray*}
m_n(t)=\frac{1}{2i\pi\,t}\int_{\gamma_t^1} [h_t(z)]^n \frac{h_t'(z)}{h_t(z)}\log(1-z) dz+\frac{1}{2i\pi\,t}\int_{\gamma_t^2} [h_t(z)]^n \frac{h_t'(z)}{h_t(z)}\log(1-z) dz.
\end{eqnarray*}
We perform the change of variables in the first integral: $h_{t,1}(z)=h_t(z)=e^{i\theta}$ then 
\begin{equation*}
i(d\theta)=\frac{h_t^\prime(z)}{h_t(z)}dz.
\end{equation*}
Since $\arg[h_t(z_t(\theta))]$ reaches its minimum at $\theta = \theta_t$, then 
\begin{equation*}
\frac{1}{2i\pi t}\int_{\gamma_t^1} [h_t(z)]^n \frac{h_t'(z)}{h_t(z)}\log(1-z) dz= - \frac{1}{2\pi t}\int_{-\beta(t)}^{\beta(t)} e^{in\theta} \log(1-h_{t,1}^{-1}(e^{i\theta})) d\theta.
\end{equation*}
Similarly we get
\begin{equation*}
\frac{1}{2i\pi t}\int_{\gamma_t^2} [h_t(z)]^n \frac{h_t'(z)}{h_t(z)}\log(1-z) dz = \frac{1}{2\pi t}\int_{-\beta(t)}^{\beta(t)} e^{in\theta} \log(1-h_{t,2}^{-1}(e^{i\theta})) d\theta,
\end{equation*}
consequently 
\begin{equation*}
m_n(t)=\frac{1}{2\pi t}\int_{-\beta(t)}^{\beta(t)} e^{in\theta} \log\Big[\frac{1-h_{t,2}^{-1}(e^{i\theta})}{1-h_{t,1}^{-1}(e^{i\theta})}\Big] d\theta.
\end{equation*}
The last part of Proposition \ref{curve} shows that:
\begin{equation*}
h_{t,1}^{-1}(e^{i\theta})=\overline{\left(\frac{h_{t,2}^{-1}(e^{i\theta})}{h_{t,2}^{-1}(e^{i\theta})-1}\right)}
\end{equation*}
implying that 
\begin{equation*}
\frac{1-h_{t,2}^{-1}(e^{i\theta})}{1-h_{t,1}^{-1}(e^{i\theta})}=|h_{t,2}^{-1}(e^{i\theta})-1|^2.
\end{equation*}
Together with $|h_{t,2}^{-1}(e^{i\theta})-1|\geq1$ yield
\begin{eqnarray*}
m_n(t)&=&\frac{1}{\pi t}\int_{-\beta(t)}^{\beta(t)} e^{in\theta} \log|h_{t,2}^{-1}(e^{i\theta})-1|\, d\theta. 
\end{eqnarray*}
Thus 
\begin{equation*}
m_{-n}(t) = \overline{m_n(t)} =  \frac{1}{\pi t}\int_{-\beta(t)}^{\beta(t)} e^{i n\theta} \log|h_{t,2}^{-1}(e^{i\theta})-1|\, d\theta, \, n \geq 1
\end{equation*}
whence we deduce that spectral measure  $\mu_t$  is given by  
\begin{equation*}
d\mu_t(\theta)=\frac{2}{t}{\bf{1}}_{[-\beta(t),\beta(t)]}(\theta)\log|h_{t,2}^{-1}(e^{i\theta})-1|\, \frac{d\theta}{2\pi} :=\rho_t(\theta)d\theta.
\end{equation*}
Moreover, $\rho_t$ is continuous since $\rho_t(-\beta(t))=\rho_t(\beta(t))=0$ which follows from $h_{t,2}^{-1}(e^{\pm i\beta_t}) = 1+ e^{\mp 2i\theta_t}$. 
\begin{nota}
Set 
\begin{equation*}
Z =  \frac{2}{z}-1
\end{equation*}
then 
\begin{equation*}
h_t(z) = h_t\left(\frac{2}{Z+1}\right) = \frac{Z-1}{Z+1}e^{tZ/2} = \tau_t^{-1}(Z)
\end{equation*}
where the last equality holds for 
\begin{equation*}
\Gamma_t:= \{\Re(Z) > 0, |\tau_t^{-1}(Z)| < 1\} = \{|z-1| < 1, |h_t(z)| < 1\}
\end{equation*}
and extends to $\overline{\Gamma_t}$ (see \cite{Biane0}). It follows that 
\begin{equation*}
h_{t,1}^{-1} = \frac{2}{\tau_t + 1}
\end{equation*}
on the closed unit disc and one derives 
\begin{align*}
\frac{2}{t}\log|h_{t,2}^{-1}(e^{i\theta})-1| &= -\frac{2}{t} \log|1- h_{t,1}^{-1}(e^{i\theta})| 
\\& = -\frac{2}{t} \log\left|\frac{\tau_t(e^{i\theta}) - 1}{\tau_t(e^{i\theta}) + 1}\right| 
\\& = -\frac{2}{t} \log\left|e^{-t\tau_t(e^{i\theta})/2}\right| = \Re[\tau(e^{i\theta})]
\end{align*} 
as stated in \cite{Biane0} p. 270.
 \end{nota}

\section{Open question; a combinatorial approach}
From a combinatorial point of view, the number 
\begin{equation*}
n^{k-1}\binom{n}{k+1}
\end{equation*}
was interpreted as the number of increasing paths having exactly $k$ steps in the Cayley graph of the symmetric group $\mathcal{S}_n$ (\cite{Levy} p. 564). In this spirit, we also figure out that the series $S_k$ displayed in the introductory part may be expressed through the so-called Riordan's polynomials $({}^rA_n)_{n \geq 0}$ for a nonnegative integer parameter $r$ (see last chapter of \cite{Riordan}). These polynomials generalize the famous Euler's polynomials and according to p. 21 in \cite{Foata}: 
\begin{align*}
S_k(u)  = \frac{{}^{k+1}A_{2k}(u)}{(k+1)!(1-u)^{2k+1}}
\end{align*}
so that 
\begin{equation*}
M_t(w) = u \sum_{k \geq 0} \frac{(-ut)^k}{k!}S_k(u)
\end{equation*}
is a generating function of Riordan's polynomials whose parameter and degree are dependent. It would be interesting to adapt Foata's summation method to this setting in order to work out the series $M_t$. 
Note nonetheless that our approach yields the representation
\begin{align*}
M_t(w) &= \frac{1}{2i\pi\, t} \int_{\gamma_t} \frac{wh_t'(z)}{1-wh_t(z)} \log(1-z) dz
\end{align*} 
for any $|w| < 1$.

\end{document}